\documentclass[12pt]{amsart}

\makeatletter
\def\newrefformat#1#2{%
  \@namedef{pr@#1}##1{#2}}
\def\prettyref#1{\@prettyref#1:}
\def\@prettyref#1:#2:{%
  \expandafter\ifx\csname pr@#1\endcsname\relax%
    \PackageWarning{prettyref}{Reference format #1\space undefined}%
    \ref{#1:#2}%
  \else%
    \csname pr@#1\endcsname{#1:#2}%
  \fi%
}
\makeatother

\newrefformat{eq}{\textup{(\ref{#1})}}
\newrefformat{cha}{Chapter \ref{#1}}
\newrefformat{sec}{Section \ref{#1}}
\newrefformat{tab}{Table \ref{#1} on page \pageref{#1}}
\newrefformat{fig}{Figure \ref{#1} on page \pageref{#1}}

\usepackage{amssymb}

\newcommand{\mynewthm}[3][dummythm]{%
  \newtheorem{#2}[#1]{#3}%
  \newrefformat{#2}{#3~\ref{##1}}%
}

\theoremstyle{plain}
\mynewthm{thm}{Theorem}
\mynewthm{prp}{Proposition}
\mynewthm{lem}{Lemma}
\mynewthm{fct}{Fact}
\mynewthm{cor}{Corollary}

\theoremstyle{definition}
\mynewthm{dfn}{Definition}
\mynewthm{conv}{Convention}
\mynewthm{ntn}{Notation}
\mynewthm{cst}{Construction}

\theoremstyle{remark}
\mynewthm{rmk}{Remark}
\mynewthm{qst}{Question}
\mynewthm{exc}{Exercise}
\mynewthm{exm}{Example}

\newcommand{\myenumlabel}[1]{\textnormal{(\roman{#1})}}

\newcounter{cycprfcnt}
\newenvironment{cycprf}%
{\begin{list}{\PackageWarning{pezz}{Label required for cycprf}}%
  {%
    \setcounter{cycprfcnt}{1}
    \setlength{\itemindent}{-0.5\leftmargin}%
    \newcommand{\cpcurr}{\myenumlabel{cycprfcnt}}%
    \newcommand{\cpnext}{\addtocounter{cycprfcnt}{1}\cpcurr}%
    \newcommand{\cpnum}[1]{\setcounter{cycprfcnt}{##1}\cpcurr}%
    \newcommand{\cpfirst}{\cpnum{1}}%
    \newcommand{\impnext}{\cpcurr{} $\Longrightarrow$ \cpnext.}%
    \newcommand{\impfirst}{\cpcurr{} $\Longrightarrow$ \cpfirst.}%
  }%
}%
{\end{list}}%

\makeatletter

\def\indsym#1#2{%
  \setbox0=\hbox{$\m@th#1x$}%
  \kern\wd0%
  \hbox to 0pt{\hss$\m@th#1\mid$\hbox to 0pt{$\m@th#1^{#2}$}\hss}%
  \lower.9\ht0\hbox to 0pt{\hss$\m@th#1\smile$\hss}%
  \kern\wd0}
\newcommand{\ind}[1][]{\mathop{\mathpalette\indsym{#1}}}

\def\nindsym#1#2{%
  \setbox0=\hbox{$\m@th#1x$}%
  \kern\wd0%
  \hbox to 0pt{\hss$\m@th#1\not$\kern1.4\wd0\hss}
  \hbox to 0pt{\hss$\m@th#1\mid$\hbox to 0pt{$\m@th#1^{#2}$}\hss}%
  \lower.9\ht0\hbox to 0pt{\hss$\m@th#1\smile$\hss}%
  \kern\wd0}

\makeatother

\DeclareMathOperator{\tp}{tp}
\DeclareMathOperator{\lstp}{lstp}
\DeclareMathOperator{\Th}{Th}
\DeclareMathOperator{\tS}{S}
\DeclareMathOperator{\dcl}{dcl}

\DeclareMathOperator{\bdd}{bdd}
\DeclareMathOperator{\Cb}{Cb}

\newcommand{\lseq}[1][]{\equiv^{\mathrm{Ls}}_{#1}}
\renewcommand{\emptyset}{\varnothing}

\def\models{\vDash}

\DeclareMathOperator{\Aut}{Aut}

\DeclareMathOperator{\dom}{dom}

\newcommand{\fP}{\mathfrak{P}}

\newcommand{\cC}{\mathcal{C}}

\newcommand{\cL}{\mathcal{L}}
\newcommand{\cM}{\mathcal{M}}

\usepackage[matrix,arrow]{xy}

\newcommand{\indP}{\ind[\fP]}

\DeclareMathOperator{\mcl}{mcl}
\DeclareMathOperator{\mcleq}{=^{mcl}}

\newcommand{\Fd}{\mathfrak{d}}
\newcommand{\Fr}{\mathfrak{r}}

\title{Lovely pairs of models: the non first order case}

\author{Itay Ben-Yaacov}

\address{Itay Ben-Yaacov\\
  Massachusetts Institute of Technology\\
  Department of Mathematics\\
  77 Massachusetts Avenue, Room 2-101\\
  Cambridge, MA 02139-4307\\
  USA}

\email{pezz@math.mit.edu}

\urladdr{http://www-math.mit.edu/\textasciitilde pezz}

\date{\today}

\thanks{This is the result of research conducted in the
  University of Illinois at Urbana-Champaign, as a part of CNRS-UIUC
  collaboration.  The author would like to thank Anand Pillay and
  Evgueni Vassiliev for hospitality and fruitful discussions}
\thanks{At the time of the writing of this paper, the author was
  a graduate student with the \'Equipe de Logique
  Math\'ematique of Universit\'e Paris VII}


\begin{document}

\begin{abstract}
  We prove that for every simple theory $T$ (or even simple thick
  compact abstract theory) there is a (unique) compact abstract theory
  $T^\fP$ whose saturated models are the lovely pairs of $T$.
  Independence-theoretic results that were proved in \cite{ppv:pairs}
  when $T^\fP$ is a first order theory are proved for the general
  case: in particular $T^\fP$ is simple and we characterise
  independence.
\end{abstract}

\maketitle

\section*{Introduction}

Lovely pairs of models of a simple first order theory were
defined in \cite{ppv:pairs}.
Under an additional assumption, namely that the equivalent conditions
of \prettyref{fct:ax} below hold, it is shown that lovely pairs
provide an elegant means for the study of independence-related
phenomena in such a theory.
This generalises a similar treatment of stable theories through the
study of beautiful pairs in \cite{poiz:paires}.

A lovely pair of models of $T$ is given by $(M,P)$ where $M \models T$ and
$P$ is a new unitary predicate defining an elementary sub-structure
with quite a few additional properties (see \prettyref{dfn:lpair}
below).
The following is proved in \cite{ppv:pairs} (the analogue for
beautiful pairs of models of a stable theory is proved in
\cite{poiz:paires}):

\begin{fct}
  Let $T$ be a complete simple first order theory.
  Then all lovely pairs of $T$ have the same first order theory $T^P$
  in the language $\cL \cup \{P\}$.
\end{fct}

This does not mean, however, that the complete first order theory
$T^P$ is meaningful.
For example, in order to use $T^P$ for the study of lovely pairs we
would like saturated models of $T^P$ to be ones.
In fact, it is proved that:

\begin{fct}
  \label{fct:ax}
  The following conditions are equivalent (for a first order simple
  theory $T$):
  \begin{enumerate}
  \item The $|T|^+$-saturated models of $T^P$ are precisely the
    lovely pairs.
  \item There is a $|T|^+$-saturated model of $T^P$ which is a lovely
    pair.
  \item The notion of elementary extension of models of $T^P$
    coincides with that of a free extension
    (\prettyref{dfn:freepair}).
  \item Every model of $T^P$ embeds elementarily in a lovely pair.
  \end{enumerate}
\end{fct}

If this holds (the ``good'' case), then $T^P$ is simple as well, and
provides an elegant
means for the study of certain independence-related properties of $T$
itself, as mentioned above.
If this fails (the ``bad'' case), then the first order theory $T^P$ is
pretty much useless.
This was noticed by Poizat in the stable case, where things
go well if and only if $T$ does not have the finite cover property;
the analogous criterion for simple theories can be argued to be the
correct analogue of non-f.c.p.\ in simple theories.

The goal of the present article is to show that in the proper context,
one can do in the ``bad'' case just the same things as in the ``good''
one.
By the previous discussion it should be clear that this cannot be done
in first order model theory, and we
need to look for a more general framework.
Such a framework, that of \emph{compact abstract theories}, or
\emph{cats}, is exposed in \cite{pezz:posmod}.
Simplicity theory is developed for cats in \cite{pezz:catsim}, but has
a few setbacks with respect to first order simplicity.
In \cite{pezz:fnctcat} we define the notion of a \emph{thick} cat
(which is still much more general than a first order theory), and
prove all basic properties of simplicity theory in this framework.

Here we prove that if $T$ is a thick simple cat (so in particular, if
$T$ is a simple first order theory), then
there exists a unique cat $T^\fP$, whose saturated models are
precisely the lovely pairs of $T$.
$T^\fP$ is also thick and simple and has a language of the same
cardinality as $T$.
In addition, there is a description (a notion close to interpretation,
defined in \cite{pezz:fnctcat}) of $T^\fP$ in $T$, which gives 
us an elegant characterisation of independence in $T^\fP$.
We also prove that if $T$ is Hausdorff, semi-Hausdorff, supersimple,
stable, stable and Robinson, or one-based, then so is $T^\fP$.

It follows from \prettyref{prp:lpair} below that
$T^\fP$ is (equivalent to) a first order theory if and only if $T$ is
and the equivalent conditions of \prettyref{fct:ax} hold.
Thus, for a first order theory $T$, the ``good'' and ``bad'' cases are
simply the first order case and the non-first-order one, respectively,
of which the former was studied in \cite{ppv:pairs}.
In the present paper, however, such considerations as whether $T^\fP$
is first order or not are hardly of any importance.

The fundamental tool is the construction of a cat from a compact
abstract
elementary category, as described in \cite{pezz:posmod}.  This tool
allows us in certain cases to fix the notion of elementary extension
as we like: since we know that things go well if and only if the
elementary extensions are the free extensions, we turn things
around and try to construct a cat where free extensions play the role
of elementary ones.
As it turns out, this is indeed one of the cases where this technique
works, and the only assumption on the original theory we actually use
is that it is a thick simple cat.

We can think of several reasons why this may be an interesting thing
to do:
First, this gives a nice and rather comprehensive set of examples of
the basic tools used in the framework of cats, and in particular of
simplicity theory.
Second, this is an additional example supporting our thesis that
simplicity in thick cats lacks nothing in comparison with simplicity in
first order theories (alas, this is not true for simplicity in
arbitrary cats).
Third, and maybe most important, thick simple cats are (or at least,
seem to be) the \emph{correct} framework for the treatment of lovely
pairs, and therefore results proved in this context should be the most
general.

Moreover, this framework allows us to state and prove results that are
either unnatural or altogether meaningless in the first order case.
Even when proving something that makes perfect sense in a first
order theory, we may use for its proof tools that would be unnatural
in the treatment of a first order theory, and this may eventually
yield a simpler or more elegant proof.
In fact, some results appearing in \cite{ppv:pairs} (notably the
preservation of one-basedness) were originally proved quite easily
in this context, and it took a bit of effort to find first order
counterparts for the ``feline'' proofs.

\bigskip

Let us give a few reminders concerning cats.
Most of this comes from \cite{pezz:posmod}.
\begin{dfn}
  Let $\cL$ be a first order language, and fix a \emph{positive
    fragment} of $\cL$, i.e., a subset $\Delta \subseteq \cL$ which is
  closed for positive boolean combinations (actually, $\cL$ is
  completely unimportant, all we want is $\Delta$).
  A \emph{formula}, unless otherwise qualified, is always a
  member of $\Delta$, and similarly for partial types.
  \\
  A \emph{universal domain} (with respect to $\Delta$) is a structure $U$
  satisfying:
  \begin{enumerate}
  \item Strong homogeneity: If $A,B \subseteq U$ are small and $f : A \to B$ is
    a $\Delta$-homomorphism (i.e., for every $\varphi \in \Delta$ and $a \in A$, $U \models \varphi(a)
    \Longrightarrow U \models \varphi(f(a))$), then $f$ extends to an automorphism of $U$  (so
    in particular, $f$ is a $\Delta$-isomorphism of $A$ and $B$).
  \item Compactness: Every small partial $\Delta$-type over $U$ which is
    finitely realised in $U$ is realised in $U$.
  \end{enumerate}
  Although this is not required by the definition, we will also assume
  that every existential formula, i.e., formula of the form
  $\exists y\,\varphi(x,y)$ where $\varphi \in \Delta$, is equivalent in $U$ to a partial $\Delta$-type.
  (If not, we can always close $\Delta$ under existential quantification
  without harming either compactness or homogeneity; this is just
  usually unnecessary.)
\end{dfn}

Saturated and strongly homogeneous models of first order theories are
one example of universal domain (with $\Delta = \cL$).
Another easy example which we will refer to later on is that of
Hilbert spaces:
\begin{exm}
  Let $H$ be the unit ball of a very large Hilbert space.
  Let $\Delta$ be the set of all formulas of the form $s \leq \|\sum_{i<n} \lambda_ix_i\|
  \leq r$ (closed under positive boolean combinations).
  Then $H$ is a universal domain w.r.t.\ $\Delta$.
\end{exm}

The negative universal theory of a universal domain
$$\Th_\Pi(U) = \{\forall\bar x\,\lnot\varphi : \varphi(\bar x) \in \Delta, U \models \forall\bar
x\,\lnot\varphi(\bar x)\}$$
has the property that the category of subsets of its e.c.\ models
has the amalgamation property (there is a little twist here, since
the notion of e.c.\ models is defined with respect to
$\Delta$-homomorphisms).
A negative universal theory having this property is called a
\emph{positive Robinson theory}.
Conversely, if $T$ is a positive Robinson theory, and in addition is
complete (i.e., the category of its e.c.\ models has the joint
embedding property), then $T = \Th_\Pi(U)$ for some universal domain
$U$; otherwise, every \emph{completion} of $T$ has a universal
domain.
Thus the giving of a universal domain is essentially the same as the
giving of a complete positive Robinson theory.
Henceforth, a \emph{theory} means a positive Robinson theory, unless
explicitly stated otherwise.

To a universal domain $U$, or to a theory $T$, we
associate type-spaces: for every set of indices $I$ we define
$\tS_I(T)$ as the set of all maximal types in $\alpha$ variables
which are consistent with $T$.
If $U$ is a universal domain for $T$ then this is the same as
$U^I/\Aut(U)$, by homogeneity.
We put a compact and $T_1$ topology on $\tS_I(T)$ by taking the closed
sets to be those defined by partial types.
If $\tS_n(T)$ is Hausdorff for every $n < \omega$ then $\tS_I(T)$ is
Hausdorff for every set $I$, and we say that $T$ is \emph{Hausdorff}.
One consequence of being Hausdorff is that the property of two tuples
to have the same type is a type-definable property.
If only the latter holds, we say that $T$ is \emph{semi-Hausdorff}.
An even weaker property is \emph{thickness}, defined in
\cite{pezz:fnctcat}, which says that indiscernibility of sequences is
type-definable.

We render the mapping $I \mapsto \tS_I(T)$ a contravariant functor in the
obvious manner: if $f : I \to J$ is any mapping, then $f^* : \tp(a_j : j
\in J) \mapsto \tp(a_{f(i)} : i \in I)$ defines a continuous mapping $f^* :
\tS_J(T) \to \tS_I(T)$.
We call this the \emph{type-space functor} of $T$, denoted $\tS(T)$.
Conversely, up to a change of language, we can reconstruct the
positive Robinson theory $T$ from $\tS(T)$ (see
\cite[Theorem~2.23]{pezz:posmod}).

Finally, in \cite[Section~2.3]{pezz:posmod} we characterise when a
class of structures equipped with
a notion of embedding has a universal domain which is also a universal
domain for a positive Robinson theory.
First, we represent such a class with concrete category $\cM$ all of
whose morphisms are injective (the embeddings); we call it an
\emph{elementary category with amalgamation} if it satisfies some
additional properties: Tarski-Vaught, elementary chain and
amalgamation (see \cite[Definition~2.27]{pezz:posmod}).
In particular, amalgamation gives us a reasonably good notion of type:
if $M$ and $N$ are models (i.e., objects of $\cM$)
and $\bar a \in M$ and $\bar b \in N$ are tuples of the same length, then
they have the same type if and only if we can embed $M$ and $N$ in a
third model $P$ such that the images of $\bar a$ and $\bar b$ in $P$
coincide.
This defines a contravariant functor $\tS(\cM)$ from sets to sets as
above.
Using this notion of types we obtain some rudimentary semantics that
allow us to state the three last requirements (see
\cite[Definition~2.32]{pezz:posmod}): that the collection of types is
not a proper class; that types of infinite tuples are determined by
the types of finite sub-tuples; and most importantly, that we can
put compact and $T_1$ topologies on each $\tS_I(\cM)$ such that its
morphisms are closed continuous mappings.
This last requirement is equivalent to saying that there is a
language $\cL$, a positive fragment $\Delta \subseteq \cL$, and a way to render
every object of $\cM$ an $\cL$-structure, such that:
\begin{enumerate}
\item The satisfaction of a $\Delta$-formula by a tuple in $M \in \cM$ is
  determine by the type of the tuple.
\item If $\Sigma$ is a set of $\Delta$-formulas, possibly in infinitely many
  variables, and $\Sigma$ is finitely realised in $\cM$, then it is
  realised in $\cM$.
\end{enumerate}
If all these requirements hold, then there exists a positive Robinson
theory $T$ (in fact $T = \Th_\Pi(\cM)$) satisfying $\tS(T) \cong \tS(\cM)$,
and e.c.\ models of $T$ embed in models of $\cM$ and vice versa.
Also, $T$ is complete if and only if $\cM$ has the joint embedding
property, and in this case a universal domain for $T$ is a universal
domain for $\cM$ in some reasonable sense.

Since we have three equivalent presentation (a
positive Robinson theory, a compact type-space functor and a compact
elementary category) of the same concept, we prefer to refer to this
concept with a generic name: \emph{compact abstract theory}, or
\emph{cat}.
The third presentation of cats is the main tool we use in the first
section.
In the fourth section we concentrate on the second approach and study
the relations between the type-space functors of our theory $T$ and of
the theory of its pairs $T^\fP$.

\bigskip

As for simplicity and independence, the thumb rule is that everything
that's true in a simple first order theory (by which we mean the main
results of \cite{kim:fork,kp:sim,hkp:can}) is true in a simple thick
cat.
Part of this is shown for arbitrary simple cats in \cite{pezz:catsim}
and the rest (in particular the extension axiom) is shown in
\cite{pezz:fnctcat} under the hypothesis of thickness.

We consider the distinction between the ``real'' sorts and the
hyperimaginary sorts immaterial: an \emph{element} is usually a real
one, but as we may adjoin any hyperimaginary sort to the original
theory it may in fact be in any such sort.
We use lowercase letters to denote elements and (possibly infinite)
tuples thereof, and uppercase letters to denote sets of such elements
or tuples (of course, any set can be enumerated into a tuple, but
sometimes it's convenient to make a conceptual distinction).

We recall that if $a$ is a tuple of elements,
or even a hyperimaginary
element, then $\bdd(a)$ (respectively $\dcl(a)$) is the collection of
all hyperimaginary elements $b$ such that $\tp(b/a)$ has boundedly
many realisations  (respectively, a unique realisation).
If $A \subseteq B$ then $A$ is \emph{boundedly closed
in $B$} if $B \cap \bdd(A) = A$.

It is also a fact that if $c \in \bdd(b)$ then $a \ind_b c$ for every
$a$, and $a \ind_b a$ if and only if $a \in \bdd(b)$.

\section{The category of $T$-pairs}

\begin{conv}
  We fix a thick simple cat $T$.
  \\
  We may consider it as a positive Robinson theory with respect to a
  positive fragment $\Delta$.  By an \emph{elementary} mapping we mean a
  $\Delta$-elementary one, that is a $\Delta$-homomorphism.
\end{conv}

We do not assume that $T$ is complete.
Therefore, instead of working inside a single universal domain for
$T$, we work with the category of e.c.\ models of $T$ (or more
precisely, of subsets thereof).
The reader should keep in mind the existence of a partial elementary
mapping between two e.c.\ models of $T$ implies that they are models
of the same completion, so we could assume that $T$ is complete
without much loss of generality.

We aim at the construction of $T^\fP$.  Our starting point is the
notions of pair and free extension/embedding:

\begin{dfn}
  \label{dfn:freepair}
  \begin{enumerate}
  \item A \emph{pair} is a couple $(A,P)$ where $A$ is a subset of
    some e.c.\ model of $T$, and $P$ is a unary predicate on $A$,
    such that $P(A)$ is boundedly closed in $A$ (i.e., $A \cap
    \bdd(P(A)) = P(A)$).
    We allow ourselves
    to omit $P$ when no ambiguity may arise, convening that it is part
    of the structure on $A$.
  \item A \emph{free embedding} of pairs $f : (A,P) \to (B,P)$ is an
    elementary embedding $f : A \to B$ such that $f(P(A)) \subseteq P(B)$ and
    $f(A) \ind_{f(P(A))} P(B)$.
    (Independence here is calculated in $B$, i.e., in any e.c.\ model
    or universal domain in which $B$ is embedded.)
  \item The \emph{free category of pairs}, $\fP$, is the category
    whose objects are pairs and whose morphisms are free embeddings.
  \end{enumerate}
\end{dfn}

\begin{lem}
  \label{lem:fremb}
  Assume that $f : (A,P) \to (B,P)$ is a free embedding.
  Then $P(f(A)) = f(P(A)) = f(A) \cap P(B)$.
\end{lem}
\begin{proof}
  Clearly $f(P(A)) \subseteq P(B) \cap f(A) = P(f(A))$, whereby
  \begin{gather*}
    f(A) \ind_{f(P(A))} P(B) \Longrightarrow P(f(A)) \ind_{f(P(A))} P(f(A)) \Longrightarrow
    P(f(A)) \subseteq \bdd(f(P(A)))
  \end{gather*}
  But then $P(f(A))$ is a subset of:
  \begin{gather*}
    \bdd(f(P(A))) \cap f(A) = f(\bdd(P(A)) \cap A) = f(P(A))
  \end{gather*}
  And the claim ensues.
\end{proof}

It follows that if the inclusion $(A,P_A) \subseteq (B,P_B)$ is a free
embedding, then $P_B$ agrees with $P_A$ on $A$, and it is legitimate
to use $P$ without further qualification.

\begin{rmk}
  One easily verifies that the identity is a free embedding, as well
  as the composition of any two free embeddings $f : A \to B$ and $g :
  B \to C$: $g(f(P(A))) \subseteq g(P(B)) \subseteq P(C)$ and $f(A) \ind_{f(P(A))}
  P(B) \Longrightarrow g(f(A)) \ind_{g(f(P(A)))} g(P(B))$, so $g(B) \ind_{g(P(B))}
  P(C) \Longrightarrow g(f(A)) \ind_{g(f(P(A)))} P(C)$ by transitivity.
  \\
  Since in a free embedding we have $f(A) \cap P(B) = f(P(A))$, we may
  usually assume that a free embedding is in fact an inclusion.
\end{rmk}

We aim to prove that $\fP$ is a compact elementary category, as
defined in \cite{pezz:posmod}.

\begin{prp}
  $\fP$ is an abstract elementary category with amalgamation
  (\cite[Definition 2.27]{pezz:posmod}).
\end{prp}
\begin{proof}
  Clearly, $\fP$ is a concrete category; we verify the properties:
  \begin{description}
  \item[Injectiveness] Every free embedding is injective.
  \item[Tarski-Vaught property] Assume that we have free inclusions $A
    \subseteq C$ and $B \subseteq C$, such that $A \subseteq B$, and we need to show that
    the inclusion $A \subseteq B$ is free as well.  It is clearly elementary,
    $P(A) \subseteq P(B)$ and $A \ind_{P(A)} P(C) \Longrightarrow A \ind_{P(A)} P(B)$.
  \item[Elementary chain property] Let $(A_i,P)$ be pairs for $i <
    \lambda$, $A_i \subseteq A_j$ freely for every $i \leq j < \lambda$, and set $(B,P) =
    \bigcup_{i<\lambda} (A_i,P)$.  By the finite character of dividing $A_i
    \ind_{P(A_i)} P(B) \Longrightarrow \bdd(P(B) \cap A_i) = P(A_i)$ for every $i$,
    so $P(B)$ is boundedly closed in $B$, and $(B,P)$ is a pair.
    Clearly $A_i \subseteq B$ is a free extension for every $i$, and $(B,P)$
    is clearly minimal as such.
  \item[Amalgamation] Assume $f : A \to B$ and $g : A \to C$ are free.
    We may embed $(A,P_A)$, $(B,P_B)$ and $(C,P_C)$ in an appropriate
    universal domain of $T$ such that $f$ and $g$ be the identity
    maps and $B \ind_A C$.
    We know that both $P_B$ and $P_C$ coincide with $P_A$ on $A$, but
    we still do not know that they coincide on $B\cap C$, so let us keep
    the distinction for a while.
    Define $D = B\cup C$ and $P_D = P_B \cup P_C$.
    Then we have:
    \begin{gather*}
      B \ind_A C \Longrightarrow B \ind_A P_C \Longrightarrow B \ind_{P_A} P_C \Longrightarrow B \ind_{P_B} P_D
    \end{gather*}
    And similarly $C \ind_{P_C} P_D$.
    Also, if $a \in D \cap \bdd(P_D)$ then either $a \in B$ or $a \in C$.
    In the former case:
    \begin{gather*}
      B \ind_{P_B} P_D \Longrightarrow a \ind_{P_B} a \Longrightarrow a \in \bdd(P_B) \Longrightarrow a \in P_B
    \end{gather*}
    and in the latter $a \in P_C$, so in either case $a \in P_D$.
    This shows that $D \cap \bdd(P_D) = P_D$, so $(D,P_D)$ is a pair, and
    the inclusions $B\subseteq D$ and $C\subseteq D$ are free.
    (It follows now by \prettyref{lem:fremb} that $P_A$, $P_B$ and
    $P_C$ are simply $P_D$ restricted to $A$, $B$ and $C$,
    respectively.)
  \end{description}
\end{proof}

Therefore we have a notion of type: we recall that $\tp^{(A,P)}(a) =
\tp^{(B,P)}(b)$ if there are free embeddings of $(A,P)$ and $(B,P)$ into
some pair $(C,P)$ such that $a$ and $b$ have the same image, and
$\tS_\alpha(\fP)$ is the set (or class, as far as we know at this point)
of types of $\alpha$-tuples in $\fP$.

The next step is to understand types:

\begin{dfn}
  Let $(A,P)$ be a pair, and $a \in A$ a tuple.
  \begin{enumerate}
  \item $a^c = \Cb(a/P(A))$ (calculated in $T$).
    \\
    As $a^c \in \bdd(P(A)) \cap \dcl(aP(A)) \subseteq \dcl(A)$, this definition
    takes place entirely within $A$; and since the canonical base over
    $P$ is invariant under free extensions, we may write it rather as
    $a^c = \Cb(a/P)$ without concerning ourselves in which specific
    pair this is taken.
  \item $\hat a = a,a^c$.
  \item The \emph{Morley class} $\mcl(a)$ is the set of pure types of
    Morley sequences (of length $\omega$) in $\tp(a/a^c)$.
  \end{enumerate}
\end{dfn}

\begin{lem}
  If $(A,P)$ and $a \in A$ are as above, then $\mcl(a)$ is the set of
  types of Morley sequences in $\tp(a/P(A))$.
\end{lem}
\begin{proof}
  Easy.
\end{proof}

\begin{lem}
  \label{lem:typchr}
  Let $(A,P)$ and $(B,P)$ be two pairs, and $a \in A$, $b \in B$ be two
  possibly infinite tuples.  Then the following are equivalent:
  \begin{enumerate}
  \item $\tp^{(A,P)}(a) = \tp^{(B,P)}(b)$ (in the sense of $\fP$).
  \item $\mcl^{(A,P)}(a) = \mcl^{(B,P)}(b)$
  \item $\mcl^{(A,P)}(a) \cap \mcl^{(B,P)}(b) \neq \emptyset$
  \item $\tp^T(\hat a^{(A,P)}) = \tp^T(\hat b^{(B,P)})$.
  \end{enumerate}
\end{lem}
\begin{proof}
  \begin{cycprf}
  \item[\impnext] $\mcl$ is invariant under free extensions.
  \item[\impnext] Morley sequences exist.
  \item[\impnext] A canonical base is in the definable closure of a
    Morley sequence.
  \item[\impfirst] We have $\hat a \in \dcl(A)$, $a^c \in \bdd(P(A))$
    and $a \ind_{a^c} P(A)$, so we may consider $(A,P)$ as a free
    extension of $(\hat a,a^c)$.  The same holds for $(\hat b,b^c)
    \subseteq (B,P)$, and now apply amalgamation.
  \end{cycprf}
\end{proof}

We need a tool that would tell us when two
types belong to the same Morley class, and this tool is the notion of
concurrently indiscernible sequences.  In fact, we prove something a
bit stronger than we actually need:

\begin{dfn}
  We say that sequences $\{(a_i^j : j < \alpha) : i < \beta\}$ are
  \emph{concurrently indiscernible over $b$} if for every $i < \beta$ and
  $j_0 < \alpha$
  the sequence $(a_i^j : j_0 \leq j < \alpha)$ is indiscernible over $b \cup
  \{a_{i'}^j : j < j_0, i' < \beta\}$ (in other word, if every tail is
  indiscernible over the union of all corresponding heads).
\end{dfn}

\begin{ntn}
  Let $\tS^{ind}_\alpha(T) \subseteq \tS_{\alpha\times\omega}(T)$ denote the set of types of
  indiscernible sequences of $\alpha$-tuples.
  In particular, $\mcl(a) \subseteq \tS^{ind}_{|a|}(T)$.
\end{ntn}

\begin{lem}
  \label{lem:cncrind}
  Assume $A = \{a_i : i < \beta\}$ are tuples, not necessarily disjoint,
  in some (e.c.) model of $T$, and $q_i \in \tS^{ind}_{|a_i|}(T)$ for
  every $i < \beta$.
  \\
  Then the following are equivalent:
  \begin{enumerate}
  \item There is some pair $(D,P)$ where $D \supseteq A$ and $q_i \in
    \mcl^{(D,P)}(a_i)$ for every $i$.
  \item There are concurrently indiscernible sequences $(b_i^j : j \leq
    \omega)$ with $b_i^\omega = a_i$ and $b_i^{<\omega} \models q_i$ for every $i$.
  \end{enumerate}
\end{lem}
\begin{proof}
  \begin{cycprf}
  \item[\impnext] For every $i < \beta$, find a Morley sequence $(b_i^j : j
    \leq \omega)$ over $P$ such that $b_i^\omega = a_i$ and $b_i^{<\omega} \models q_i$.
    Write $B^k = \{b_i^k : i < \beta\}$.
    \\
    We now give a construction by induction on $k < \omega$.
    At the beginning of the $k$th step
    we assume that $(b_i^j : k \leq j \leq \omega)$ is a Morley sequence over
    $PB^{<k}$ for every $i$.  During the step we may move $(b_i^j : k
    \leq j < \omega)$ around a bit in order to obtain the same thing for $k
    + 1$ without moving $AB^{<k}$, nor changing
    $\tp(b_i^{<\omega}/a_iPB^{<k})$.  From this point onward, $B^k$ is
    fixed as well.
    \\
    We may assume for every $i$ that $b_i^{<\omega} \ind_{a_iPB^{<k}} A$,
    whereby $b_i^{<\omega} \ind_{PB^{<k}} A$.  We may further assume that
    $\{b_i^{<\omega} : i < \beta\} \cup \{A\}$ is a $PB^{<k}$-independent set.
    At this point we fix $B^k = \{b_i^k: i < \beta\}$ for the rest of the
    construction, and observe that $B^k \ind_{PB^{<k}} A$.
    \\
    We now work for each $i$ separately: we observe that $a_i
    \ind_{PB^{<k}b_i^k} B^k$ by the previous paragraph and that
    $(b_i^j : k < j \leq \omega)$ is a Morley sequence over $PB^{<k}b_i^k$
    with $a_i = b_i^\omega$.  Therefore there is an automorphism fixing
    $a_iPB^{<k}b_i^k$ that when applied to $(b_i^j : k < j < \omega)$
    gives an $PB^{\leq k}$-indiscernible sequence, and in fact a Morley
    sequence over $PB^{\leq k}$, as required.  We now fix
    $\tp(b_i^{<\omega}/a_iPB^{\leq k})$, and the construction continues.
    \\
    At the end we obtain concurrently indiscernible Morley sequences
    over $P$ with the required types.
  \item[\impfirst] Let $D = AB^{<\omega}$ and $P(D) = D \cap \bdd(B^{<\omega})$.
    Then $(D,P)$ is a pair.
    \\
    Since $(b_i^j : k \leq j \leq \omega)$ is $B^{<k}$-indiscernible, we have
    $a_i = b_i^\omega \ind_{b_i^{[k,\omega)}} B^{<k}$ for every $k < \omega$,
    whereby $a_i \ind_{b_i^{<\omega}} P(D)$, and $a_i^c =
    \Cb(b_i^\omega/b_i^{<\omega})$.
    \\
    Since $(b_i^j : j \leq \omega)$ is an indiscernible sequence it is a
    Morley sequence over $a_i^c$, and $\tp(b_i^{<\omega}) \in
    \mcl^{(D,P)}(a_i)$.
  \end{cycprf}
\end{proof}

\begin{ntn}
  \begin{enumerate}
  \item For $p \in \tS(\fP)$ define $\mcl(p)$ as $\mcl(a)$ for any $a
    \models p$: by \prettyref{lem:typchr} this is well defined.
    Similarly, for a set $F \subseteq \tS_\alpha(\fP)$, we define $\mcl(F) = \bigcup_{p \in
      F} \mcl(p) = \bigcup_{\tp^\fP(a) \in F} \mcl(a)$.
  \item For tuples $a_{<\omega}$ and $b_{<\omega}$ (in an e.c.\ model of $T$)
    such that all $a_i$
    and $b_i$ are of the same length $\alpha$, say that $a_{<\omega} \mcleq
    b_{<\omega}$ if there exist $a_\omega = b_\omega$ such that $(a_i : i
    \leq \omega)$ and $(b_i : i \leq \omega)$ are concurrently indiscernible.
    Since $T$ is thick, this property is defined by a partial type
    $r_\alpha(x_{<\omega},y_{<\omega})$.
    We usually omit the subscript $\alpha$ since it can be deduced from the
    context.
  \end{enumerate}
\end{ntn}

Then \prettyref{lem:cncrind} gives:

\begin{cor}
  \label{cor:mcleq}
  \begin{enumerate}
  \item If $q,q' \in \tS_{\alpha\times\omega}(T)$, then $x_{<\omega} \mcleq y_{<\omega} \land
    q(x_{<\omega}) \land q'(y_{<\omega})$ is consistent if and only if there is $p
    \in \tp_\alpha(\fP)$ such that $q,q' \in \mcl(p)$.
  \item Let $p \in \tS_\alpha(\fP)$ and $q \in \mcl(p)$.
    Then the partial type $\exists y_{<\omega} \,[r(y_{<\omega},x_{<\omega}) \land q(y_{<\omega})]$
    defines the set $\mcl(p)$, which is in particular
    closed.
    (An existential quantification on a partial type is equivalent to
    a partial type, by compactness.)
  \end{enumerate}
\end{cor}

In particular, we may identify $\mcl(p)$ with the partial type $\exists
y_{<\omega} \, [r(x_{<\omega},y_{<\omega}) \land q(y_{<\omega})]$ for any $q \in \mcl(p)$.

In addition if $a_i$ and $b_i$ are $\alpha$-tuples for $i < \omega$, then
$a_{<\omega} \mcleq b_{<\omega}$ if and only if $a'_{<\omega} \mcleq b'_{<\omega}$
for every possible choice of corresponding sub-tuples $a_i' \subseteq a_i$,
$b_i' \subseteq b_i$.
It follows that $\fP$-types satisfy the local character, namely the
types of two infinite tuples are equal if and only if the types of
every two corresponding finite sub-tuples are equal.
We conclude that $\tS(\fP)$ is a set type-space functor.

It is time now to define a language for $\fP$:

\begin{dfn}
  \label{dfn:LP}
  Let $\varphi(x_{<k}) \in \Delta$, where each $x_i$ is an $n$-tuple.
  We define $R_\varphi$ as the set of all $p
  \in \tS_n(\fP)$ such that there is $q(x_{<\omega}) \in \mcl(p)$ satisfying
  $\varphi(x_{<k})$ (that is to say that $\mcl(p)$ is consistent with
  $\varphi$).
  \\
  We interpret $R_\varphi$ as an $n$-ary predicate on pairs in the obvious
  way: if $(A,P) \in \fP$ and $a \in A^n$ then $(A,P) \models R_\varphi(a) \Longleftrightarrow
  \tp^{(A,P)}(a) \in R_\varphi$.
  \\
  We define $\cL^\fP$ as the set of all such predicates, so $|\cL^\fP|
  = |\cL|$.
  We also define $\Delta^\fP = \Delta_0(\cL^\fP)$, that is the positive
  quantifier-free formulas in $\cL^\fP$.
\end{dfn}

\begin{rmk}
  We cheat a bit, since $R_\varphi$ depends not only on $\varphi$ but
  on the actual decomposition of its free variables into $k$
  $n$-tuples, but we are just going to consider that this information
  is contained in $\varphi$.
\end{rmk}

Ordinarily, the set of quantifier-free formulas is closed under
conjunction, disjunction and change of variables.
We recall that if $f : n \to m$ is a map and $\varphi(x_{<n})$ a formula,
then $\psi(y_{<m}) = \varphi(y_{f(0)},\ldots,y_{f(n-1)})$ is obtained from $\varphi$
through a change of variables by $f$, and we may also write $\psi =
f_*(\varphi)$.
However, in this particular language, the finite disjunction and
change of variables are not necessary:

\begin{lem}
  \begin{enumerate}
  \item Let $R_\varphi(x^{<n})$ be an $n$-ary predicate in this language,
    where $\varphi(x^{<n}_0,\ldots,x^{<n}_{k-1}) \in \Delta$.
    Let $y^{<m}$ be another tuple of variables and $f : n \to m$ a map,
    and let us convene that by $y^{f(<n)}$ we mean the tuple
    $y^{f(0)},\ldots,y^{f(n-1)}$.
    Then the formula $f_*(R_\varphi)(y^{<m}) = R_\varphi(y^{f(<n)})$ is
    equivalent to $R_\psi(y^{<m})$ where $\psi(y^{<m}_0,\ldots,y^{<m}_{k-1}) =
    \varphi(y^{f(<n)}_0,\ldots,y^{f(<n)}_{k-1})$.
  \item $R_\varphi \lor R_\psi$ is equivalent to $R_{\varphi\lor\psi}$.
  \end{enumerate}
\end{lem}

This means that every $n$-ary $\Delta^\fP$-formula is equivalent to a
conjunction of $R_\varphi$-predicates, as finite disjunctions and changes
of variables can be transferred to $\varphi$, and similarly for partial
$\Delta^\fP$-types.

\begin{lem}
  \label{lem:Rmcl}
  \begin{enumerate}
  \item Let $\rho(x_{<\omega})$ be a partial $\Delta$-type, which we may assume
    to be closed under finite conjunctions, and let $R_\rho(x) =
    \bigwedge_{\varphi(x_{<k}) \in \rho} R_\varphi(x)$.
    Then $p \vdash R_\rho$ if and only if $\mcl(p) \land \rho$ is consistent.
  \item Conversely, if $R_\varphi(x) \in \cL^\fP$, then $\mcl(R_\varphi)$ is
    defined by the partial type $\exists y_{<\omega} \mcleq x_{<\omega} \,
    \varphi(y_{<k})$; and if $\rho(x) = \bigwedge_{i < \lambda} R_{\varphi_i}(x)$ then $\mcl(\rho) =
    \bigwedge_{i < \lambda} \mcl(R_{\varphi_i})$.
  \end{enumerate}
\end{lem}
\begin{proof}
  \begin{enumerate}
  \item Since $\mcl(p)$ is a closed set, we have that $\mcl(p)$ is
    consistent with $\rho$ if and only if $\mcl(p)$ is finitely
    consistent with $\rho$ if and only if $p \vdash R_\rho$.
  \item Directly by \prettyref{cor:mcleq}
  \end{enumerate}
\end{proof}

So let us see now what can be expressed in this language.
All the following are easily verifiable:

\begin{itemize}
\item Any complete $\fP$-type: for any $p \in \tS(\fP)$ is defined by
  $R_{\mcl(p)}$.
\item Equality: $x = y$ is defined by $R_{x_{<\omega} = y_{<\omega}}(x,y)$.
\item Existential quantification: if $\rho(x,y)$ is a partial
  $\Delta^\fP$-type, and $\mcl(\rho)$ is defined by $\rho'(x_{<\omega},y_{<\omega})$,
  then $\exists y \, \rho(x,y)$ is defined by $R_{\exists y_{<\omega} \, \rho'}$.
  Therefore, our assumption that $\Delta$ eliminates the existential
  quantifier (for $T$) implies that so does $\Delta^\fP$ (for $\fP$).
\item Any $\Delta$-formula $\varphi(x)$: this is just $R_{\varphi(x_0)}$.
\item $x \in P$: take $R_{x_0 = x_1}$.
\item Indiscernibility of sequences:
  write $X = x^{<\omega}$, and let $\rho(X_{<\omega})$ say that $(x_{<\omega}^j : j <
  \omega)$ is an indiscernible sequence (which is possible since $T$ is
  thick).
  Then $R_\rho(X)$ says that $X = (x^j : j < \omega)$ is an indiscernible
  sequence in the sense of $\fP$.
  This shows that $\fP$ is thick.
\item Equality of types: if $T$ is semi-Hausdorff then $\rho(x_{<\omega},y_{<\omega}) = \exists z_{<\omega} \,
  [x_{<\omega} \mcleq z_{<\omega} \equiv y_{<\omega}]$ is a partial type, and
  $R_\rho(x,y)$ defines the property $x \equiv y$, so $T^\fP$ is
  semi-Hausdorff as well.
\item If inequality is positive in $T$, we may say that $x \notin P$, by
  $R_{x_0 \neq x_1}$ (this can be improved).
\end{itemize}

The last thing to prove is that this logic is compact.

\begin{lem}
  Let $\Sigma(X)$ be some partial $\Delta^\fP$-type, where $X$ is a possibly
  infinite tuple.
  Then $\Sigma$ is realised in $\fP$ if and only if it is finitely
  realised in $\fP$.
\end{lem}
\begin{proof}
  Write $\mcl(\Sigma)(X_{<\omega}) = \bigwedge_{x \subseteq X, \varphi(x) \in \Sigma}
  \mcl(\varphi)(x_{<\omega})$.
  Then $\Sigma$ is realised if and only if $\mcl(\Sigma)$ is consistent if and
  only if $\mcl(\Sigma)$ is finitely consistent if and only if $\Sigma$ is
  finitely realised.
\end{proof}

And we conclude:

\begin{dfn}
  $T^\fP = \Th_{\Pi^\fP}(\fP)$ is the negative universal theory of
  pairs in this language.
\end{dfn}

\begin{thm}
  \label{thm:TPcat}
  $T^\fP$ is a thick positive Robinson theory in $\Delta^\fP$, and
  $\tS(\fP) = \tS(T^\fP)$.
  \\
  If $T$ is semi-Hausdorff or Hausdorff, then so is $T^\fP$.
  \\
  If $T$ is complete then so is $T^\fP$; otherwise, there is a
  bijection between completions of $T$ and $T^\fP$.
\end{thm}
\begin{proof}
  We showed that $\cL^\fP$ is a language for $\fP$ which can define
  complete types and satisfies weak compactness.
  Thus, by \cite{pezz:posmod}, $T^\fP$ is a positive Robinson theory
  in $\Sigma^\fP$, and $\tS(\fP) = \tS(T^\fP)$, where $\Sigma^\fP$ is the set
  of positive existential $\cL^\fP$-formulas.
  However, as we proved that the language $\Delta^\fP$ eliminates the
  existential quantifier, we can replace $\Sigma^\fP$ with $\Delta^\fP$.
  \\
  We also already proved that $T^\fP$ is thick, and if $T$ is
  semi-Hausdorff then so is $T^\fP$.
  If $T$ is Hausdorff, and $p \neq p' \in \tS_n(T^\fP)$, then $\mcl(p) \cap
  \mcl(p') = \emptyset$, so they can be separated by open sets.  In other
  words, there are partial types $\rho(x_{<\omega})$ and $\rho'(x_{<\omega})$,
  inconsistent with $\mcl(p)$ and $\mcl(p')$, respectively, such that
  $\models \rho \lor \rho'$.
  Then $R_\rho$ and $R_{\rho'}$ are inconsistent with $p$ and $p'$,
  respectively, and $\fP \models R_\rho \lor R_{\rho'}$, so $p$ and $p'$ are also
  separated by open sets.
  \\
  If $(A,P)$ and $(B,P)$ are two pairs, and $A$ and $B$ embed in e.c.\
  models of the same completion of $T$, then we can amalgamate the two
  pairs over $(\emptyset,\emptyset)$.
  On the other hand, if $A$ and $B$ belong to distinct completions of
  $T$ then we cannot embed them in a single e.c.\ model.
  Therefore the completions of $T$ are in bijection with those of
  $T^\fP$.
\end{proof}

\begin{conv}
  We shall work in a universal domain $U^\fP$ for (a completion of)
  $T^\fP$.
\end{conv}

\section{Lovely pairs}

Since the origin of the theory of pairs is in lovely ones, we need to
say something about them.

\begin{dfn}
  \label{dfn:lpair}
  Let $\kappa > |T|$.  A pair $(M,P)$ is $\kappa$-lovely if:
  \begin{enumerate}
  \item For every $A \subseteq M$ such that $|A| < \kappa$, and for every type $p
    \in S^T(A)$, there is $a \models p$ in $M$ with $a \ind_A P(M)$.
  \item For every $A \subseteq M$ with $|A| < \kappa$ and every type $p \in \tS(A)$
    which does not divide over $P(A)$, there is $a \models p$ in $P(M)$.
  \end{enumerate}
\end{dfn}

\begin{dfn}
  A set $A$ in $U^\fP$ is \emph{free} if $A \ind_{P(A)} P$.
  \\
  This means that $(A,P)$ is freely embedded in the universal domain,
  so it determines $\tp^\fP(A)$.
\end{dfn}

\begin{prp}
  \label{prp:lpair}
  Let $\kappa > |T|$.  Then a pair $(M,P)$ is a $\kappa$-saturated model of
  $T^\fP$ if and only if it is $\kappa$-lovely.
\end{prp}
\begin{proof}
  Let $(M,P)$ be $\kappa$-saturated, and we want to prove that it is
  $\kappa$-lovely:
  \begin{enumerate}
  \item Assume that $A \subseteq M$, $|A| < \kappa$, and $a$ is some element
    possibly outside $M$.
    As we are only interested in $\tp^T(a/A)$, we may assume that $a
    \ind_A M$.
    \\
    Set $D = aM$, $P(D) = D \cap \bdd(P(M))$.
    Then $(D,P)$ is a pair, and a free extension of $(M,P)$.
    By saturation, there is an element $a' \in M$ such that
    $\tp^\fP(a/A) = \tp^\fP(a'/A)$.
    Define $b = (aA)^c$, and $b' = (a'A)^c$, so:
    \begin{gather*}
      b,b' \in \bdd(P(D)) = \bdd(P(M)) \subseteq \bdd(M)
    \end{gather*}
    (In fact, since $M$ is $|T|^+$-saturated we have $\bdd(M) =
    \dcl(M)$, but this is not used here.)
    From $\tp^\fP(a/A) = \tp^\fP(a'/A)$ we obtain $\tp^T(aAb) =
    \tp^T(a'Ab')$.
    Then we have $a'A \ind_{b'} P$, but also:
    \begin{gather*}
      a \ind_A M \Longrightarrow a \ind_A b \Longrightarrow a' \ind_A b' \Longrightarrow a' \ind_A P,
    \end{gather*}
    as required.
  \item Assume that $A \subseteq M$, $|A| < \kappa$, and $a \ind_{P(A)} A$.  We
    may assume that $a \ind_{P(A)} M$ so $a \ind_{P(M)} M$.
    Let $D = aM$ as above, but define $P(D) = D \cap \bdd(aP(M))$.
    Then $(D,P)$ is a free extension of
    $(M,P)$, and $\tp^\fP(a/A)$ is realised in $M$.
  \end{enumerate}
  For the converse, assume that $(M,P)$ is $\kappa$-lovely.  Assume that
  $A \subseteq M$, $|A| < \kappa$ and $a$ is an element of some free extension
  $(N,P)$ of $(M,P)$.  Write $\mu = |A| + |T| < \kappa$.
  \\
  We may find $B \subseteq P(M)$ such that $|B| \leq \mu < \kappa$ and $A \ind_B
  P(M)$.  Replacing $A$ with $A \cup B$ we may assume that $A$ is free.
  Now find $C \subseteq P(N)$ such that $|C| \leq \mu < \kappa$ and $a \ind_{AC}
  P(N)$.  Since $A$ is free in $M$ it is also free in $N$, so $A
  \ind_{P(A)} C$ and therefore there is $C' \subseteq P(M)$ with $C' \equiv_A C$.
  Then there is $a' \in M$ such that $a'C' \equiv_A aC$ and $a' \ind_{AC'}
  P(M)$. Then $aCA$ and $a'C'A$ are both free sets with $aCA \equiv
  a'C'A$, whereby $aCA \equiv^\fP a'C'A$, so in particular $a \equiv^\fP_A a'$
  as required.
\end{proof}

\begin{cor}
  Every pair $(A,P)$ has a free extension to a $\kappa$-lovely pair.
\end{cor}
\begin{proof}
  Just embed it freely in a sufficiently saturated model of $T^\fP$.
\end{proof}

\begin{rmk}
  Assume that $T$ is complete, and consider the language $\cL_P = \cL
  \cup \{P\}$.  Then every two $|T|^+$-lovely pairs are elementarily
  equivalent in this language, and any two free sets of cardinality
  $\leq |T|$ with the same $\cL_P$-diagram have the same type (this
  generalises results in \cite{poiz:paires,ppv:pairs}).
  \\
  Indeed, since two such sets have the same $\fP$-type, they
  correspond by an infinite back-and-forth in saturated structures.
  In particular, since the empty set is free, we have the elementary
  equivalence.
  \\
  However, this is just a special case of a more general observation:
  taking any two saturated models of a cat (or in fact, any two
  equi-universal homogeneous structures), and taking any relational
  language whose $n$-ary predicates are interpreted as subsets of
  $\tS_n$ (without any topological requirement), then they are
  elementarily equivalent in this language.  Of course, they have no
  reason to be saturated as models of their first-order theory, and
  when they are not, this first-order theory is rather meaningless.
\end{rmk}

\section{Independence in $T^\fP$}

\subsection{Simplicity}

We prove that $T^\fP$ is simple and characterise independence.

\begin{prp}
  \label{prp:indPeq}
  The following conditions are equivalent for (possibly infinite)
  tuples $a,b,c$ in $\fP$:
  \begin{enumerate}
  \item Whenever $(a_ib_ic_i : i < \omega) \models \mcl(abc)$, then $b_{<\omega}
    \ind_{a_{<\omega}} c_{<\omega}$.
  \item There exist $(a_ib_ic_i : i < \omega) \models \mcl(abc)$ such that $b_{<\omega}
    \ind_{a_{<\omega}} c_{<\omega}$.
  \item $(abc)^c \in \bdd((ab)^c,(ac)^c)$ and $\widehat{ab} \ind_{\hat a}
      \widehat{ac}$.
  \item \label{prp:indPeq:indPhat} $b \ind_{aP} c$ and $\widehat{ab}
    \ind_{\hat a} \widehat{ac}$.
  \item \label{prp:indPeq:indPc} $b \ind_{aP} c$ and $(ab)^c
    \ind_{a^c} (ac)^c$.
  \end{enumerate}
\end{prp}
\begin{proof}
  \begin{cycprf}
  \item[\impnext] Clear.
  \item[\impnext]
    We are given $(a_ib_ic_i : i < \omega) \models \mcl(abc)$ such that
    $b_{<\omega} \ind_{a_{<\omega}} c_{<\omega}$.
    This is a Morley sequence in $\tp(abc/(abc)^c)$, and we may assume
    that $abc = a_0b_0c_0$.
    In particular, the sequence $(a_i : 0 < i < \omega)$ is a Morley
    sequence over $\hat a$, indiscernible over $\widehat{ab}$, whereby
    $\widehat{ab} \ind_{\hat a} a_{<\omega}$, so:
    \begin{align*}
      b_{<\omega} \ind_{a_{<\omega}} c_{<\omega} & \Longrightarrow \widehat{ab}
      \ind_{a_{<\omega}} \widehat{ac} \Longrightarrow \widehat{ab}
      \ind_{\hat a} \widehat{ac}
    \end{align*}
    We also know that $ab \ind_{(ab)^c} a_{[1,\omega)}b_{[1,\omega)}$, and
    that $ac \ind_{(ac)^c} (abc)^c \Longrightarrow ac \ind_{(ab)^c(ac)^c} (abc)^c$.
    We obtain:
    \begin{align*}
      b_{<\omega} \ind_{a_{<\omega}} c_{<\omega} & \Longrightarrow b \ind_{a(ab)^c}
      a_{[1,\omega)}b_{[1,\omega)}c_{<\omega} \Longrightarrow b \ind_{a(ab)^c} c(abc)^c
      \\
      & \Longrightarrow abc \ind_{(ab)^c(ac)^c} (abc)^c
    \end{align*}
    Since $\tp(abc/(abc)^c)$ does not divide over $(ab)^c,(ac)^c$, we
    obtain $(abc)^c = \Cb(abc/(abc)^c) \in \bdd((ab)^c,(ac)^c)$.
  \item[\impnext]
    \begin{align*}
      \widehat{ab} \ind_{\hat a} \widehat{ac} & \Longrightarrow b
      \ind_{a(ab)^c(ac)^c} c \Longrightarrow b \ind_{a(abc)^c} c
    \end{align*}
  \item[\impnext]
    \begin{align*}
      \widehat{ab} \ind_{\hat a} \widehat{ac} & \Longrightarrow (ab)^c \ind_{\hat
        a} (ac)^c \Longrightarrow (ab)^c \ind_{a^c} (ac)^c
    \end{align*}
  \item[\impfirst]
    We know that $(a_ib_ic_i : i < \omega)$ is a Morley sequence over
    $(abc)^c$, so $a_ib_ic_i \ind_{(abc)^c} a_{\neq i}b_{\neq i}c_{\neq i}$ for
    all $i < \omega$.
    Then:
    \begin{align*}
      b \ind_{a(abc)^c} c & \Longrightarrow b \ind_{a(ab)^c} c(abc)^c
      \\
      a_ib_ic_i \ind_{(abc)^c} a_{\neq i}b_{\neq i}c_{\neq i} & \Longrightarrow b_i
      \ind_{a_i(ab)^c} a_{\neq i}b_{<i}c_{<\omega} \Longrightarrow b_i \ind_{a_{<\omega}(ab)^c} b_{<i}c_{<\omega}
    \end{align*}
    By induction on $i$ we obtain $b_{<i} \ind_{a_{<\omega}(ab)^c}
    c_{<\omega}$ for all $i$, whereby $b_{<\omega} \ind_{a_{<\omega}(ab)^c}
    c_{<\omega}$.
    Finally, $(ab)^c \ind_{a^c} (ac)^c$ gives us:
    \begin{align*}
      a_{<\omega}c_{<\omega} \ind_{(ac)^c} (ab)^c & \Longrightarrow a_{<\omega}c_{<\omega}
      \ind_{a^c} (ab)^c \Longrightarrow c_{<\omega} \ind_{a_{<\omega}} (ab)^c
      \\
      & \Longrightarrow b_{<\omega} \ind_{a_{<\omega}} c_{<\omega}
    \end{align*}
    As required.
  \end{cycprf}
\end{proof}

\begin{dfn}
  If any of the equivalent conditions in \prettyref{prp:indPeq} holds
  we say that $b \indP_a c$.
\end{dfn}

\begin{rmk}
  Conditions \ref{prp:indPeq:indPhat} and \ref{prp:indPeq:indPc} of
  \prettyref{prp:indPeq} were proposed independently, in some form or
  another, by all three authors of \cite{ppv:pairs} as candidates for
  independence in $T^\fP$ (in the case where $T^\fP$ is first order).
\end{rmk}

\begin{thm}
  \label{thm:TPsim}
  $T^\fP$ is simple, and $b \indP_a c$ if and only if $\tp^\fP(b/ac)$
  does not divide over $a$.
\end{thm}
\begin{proof}
  We need to prove that $\indP$ is an independence relation.
  \\
  By \prettyref{prp:indPeq}, if $abc \in U^\fP$ and $(a_ib_ic_i : i <
  \omega) \models \mcl(abc)$, then $b \indP_a c \Longleftrightarrow b_{<\omega} \ind_{a_{<\omega}}
  c_{<\omega}$.
  This gives immediate proofs for all the properties of an
  independence relation, with the exception of the independence
  theorem for Lascar strong types, which we treat separately.
  \\
  So assume that $\lstp^\fP(b_0/a) = \lstp^\fP(b_1/a)$, $c_0 \indP_a
  c_1$, and $c_i \indP_a b_i$ for $i < 2$.
  Then in particular $\widehat{ab_0} \lseq[\hat a] \widehat{ab_1}$,
  and we also have $\widehat{ab_i} \ind_{\hat a} \widehat{ac_i}$ for
  $i < 2$ and $\widehat{ac_0} \ind_{\hat a} \widehat{ac_1}$.
  By the independence theorem in $T$ we can find $b,d \ind_{\hat a}
  \widehat{ac_0}\widehat{ac_1}$ such that $b,d \equiv_{\widehat{ac_i}}
  b_i,(ab_i)^c$.
  \\
  As $(ac_0c_1)^c \in \bdd((ac_0)^c(ac_1)^c)$ we have in fact $b,d \ind_{\hat a}
  \widehat{ac_0c_1}$, so $d \ind_{a^c} \widehat{ac_0c_1} \Longrightarrow d
  \ind_{(ac_0c_1)^c} ac_0c_1$.
  We may therefore realise $d$ in $P$, and then realise $b$ in $U^\fP$
  such that $b \ind_{\widehat{ac_0c_1},d} P \Longrightarrow b \ind_{ad} c_0c_1P$.
  In particular, $ab \ind_d P$ (since $d \in P$ and $a^c \in \dcl(d)$),
  and $d = (ab)^c$.
  \\
  For $i < 2$, we get $b \ind_{a(ab)^c} c_iP \Longrightarrow abc_i \ind_{(ab)^c(ac_i)^c}
  P$.
  Recall that $(ab_ic_i)^c \in \bdd((ab_i)^c(ac_i)^c)$, so $ab_ic_i
  \ind_{(ab_i)^c(ac_i)^c} P$ as well.
  This along with $\widehat{ab} \equiv_{\widehat{ac_i}} \widehat{ab_i}$
  yields  $b \equiv^\fP_{ac_i} b_i$.
  \\
  We know that $\widehat{ab} \ind_{\hat a} \widehat{ac_0c_1}$, and
  $b \ind_{a(ab)^c} c_0c_1P \Longrightarrow b \ind_{aP} c_0c_1$, so $b \indP_a
  c_0c_1$, as required.
\end{proof}

Notice that in the proof of the independence theorem, we used the
assumption $\lstp^\fP(b_0/a) = \lstp^\fP(b_1/a)$ only to conclude that
$\widehat{ab_0} \lseq[\hat a] \widehat{ab_1}$.
This implies that:

\begin{cor}
  For every $a \in U^\fP$, $\bdd^\fP(a)$ is $\fP$-interdefinable with
  $\bdd(\hat a)$.
  \\
  This still holds even if we consider hyperimaginary sorts of $T^\fP$
  that are not inherited from $T$.
\end{cor}
\begin{proof}
  One inclusion is clear.
  For the other, let $b_E \in \bdd^\fP(a)$ be a hyperimaginary.
  We saw that $\tp^\fP(b/\bdd(\hat a))$ is an amalgamation base, so it
  is equivalent to $\tp^\fP(b/\bdd^\fP(a))$ and therefore implies
  $\tp^\fP(b/b_E)$.
  Then every automorphism of $U^\fP$ that fixes $\bdd(\hat a)$ sends
  $b$ to another realisation of $\tp^\fP(b/b_E)$, and therefore fixes
  $b_E$, so $b_E \in \dcl^\fP(\bdd(\hat a))$.
\end{proof}

\begin{cor}
  \label{cor:TPssim}
  If $T$ is supersimple, then so is $T^\fP$.
\end{cor}
\begin{proof}
  Let $a$ be a singleton and $B = \{b^i : i < \alpha\}$ a set in $U^\fP$.
  Let $(a_j,B_j : j < \omega) \models \mcl(a,B)$ in $U$, and extend
  this   to a similar $2\omega$-sequence $(a_j,B_j : j < 2\omega)$.
  By supersimplicity, there are $n < \omega$ and $I \subseteq \alpha$ finite such
  that $a_\omega \ind_{a_{<n},b^{\in I}_{\in[0,n)\cup[\omega,2\omega)}}
  a_{<\omega}B_{<2\omega}$.
  \\
  Then for every $m \in [n,\omega)$ we have $a_\omega \ind_{a_{<m},b^{\in I}_{[0,m)\cup[\omega,2\omega)}}
  B_{[0,m)\cup[\omega,2\omega)}$, and by removing the segment $[m,\omega)$ we obtain
  $a_m \ind_{a_{<m},b^{\in I}_{<\omega}} B_{<\omega}$.
  On the other hand, increasing $I$ somewhat, though keeping it
  finite, we may also   assume that $a_{<n} \ind_{b^{\in I}_{<\omega}}
  B_{<\omega}$.
  Combined with the previous observations, an easy induction gives
  $a_{<m} \ind_{b^{\in I}_{<\omega}} B_{<\omega}$ for every $m \in [n,\omega)$,
  whereby $a_{<\omega} \ind_{b^{\in I}_{<\omega}} B_{<\omega}$.
  \\
  We conclude that $a \indP_{b^{\in I}} B$, with $|I| < \omega$, as required.
\end{proof}

\begin{rmk}
  The approach we take here for the proof of simplicity and the
  characterisation of independence in $T^\fP$ is completely different
  than that which appears in \cite{ppv:pairs}.
  The basic improvement is in the equivalence $b
  \indP_a c \Longleftrightarrow b_{<\omega} \ind_{a_{<\omega}} c_{<\omega}$ which does not appear
  there.
  Given this equivalence, all that is left to show is the independence
  theorem, which then gives us at once the simplicity of $T^\fP$, the
  characterisation of dividing, and the characterisation of Lascar
  strong types.
  \\
  In fact, not knowing what hyperimaginary sorts in $T^\fP$ look like,
  the only way we know how to prove that $\bdd^\fP(a) =
  \dcl^\fP(\bdd(\hat a))$ is through the independence theorem, so
  might just as well obtain the other results at the same time.
\end{rmk}

\textbf{Added in proof:}
Recent results suggest that the ``obvious'' definition of
supersimplicity is too strong for general cats (more precisely, for
those where the property $x \neq y$ is not positive).
A better (and more permissive) definition appears in \cite{pezz:morley}
for Hausdorff cats.
The analogue \prettyref{cor:TPssim} for this definition is true,
although it does not seem possible to prove it solely with
the tools introduced in \cite{pezz:morley}.

\subsection{Stability}

We recall:

\begin{dfn}
  \begin{enumerate}
  \item $T$ is \emph{$\lambda$-stable} if $|\tS_n(A)| \leq \lambda$ for every set
    $|A| \leq \lambda$.
  \item $T$ is \emph{stable} if it is $\lambda$-stable for some $\lambda$.
  \item $T$ is \emph{superstable} if it is $\lambda$-stable for every $\lambda \geq
    2^{|T|}$.
  \end{enumerate}
\end{dfn}

One can prove along the lines of the classical proof:

\begin{fct}
  \label{fct:stab}
  Let $T$ be any cat.
  \begin{enumerate}
  \item $T$ is stable if and only if $T$ is $\lambda^{|T|}$-stable for
    every $\lambda$.
  \item $T$ is superstable if and only if it is stable and supersimple.
  \end{enumerate}
\end{fct}

\begin{thm}
  If $T$ is stable or superstable, then so is $T^\fP$.
\end{thm}
\begin{proof}
  Assume that $T$ is stable, and count $\fP$-types over a set $A$.
  Fix a sequence $(A_i : i < \omega) \models \mcl(A)$: for every
  $a$, $\tp^\fP(a/A)$ is determined by $\tp(a_{<\omega}/A_{<\omega})$, for any
  $a_{<\omega}$ such that $(a_iA_i : i < \omega) \models \mcl(aA)$ (and such
  $a_{<\omega}$ always exists).
  By stability: $|S^\fP_\lambda(A)| \leq |S^T_{\lambda + \omega}(A_{<\omega})| \leq (|A| +
  \omega)^{\lambda + |T|}$, so $|A| = \mu^{|T|} \Longrightarrow |S^\fP_{|T|}(A)| = |A|$.
  \\
  The result of superstable follows from \prettyref{fct:stab} and
  \prettyref{cor:TPssim}.
\end{proof}

\subsection{One-basedness}

We recall:

\begin{dfn}
  A simple cat $T$ (not necessarily thick) is \emph{one-based}
  if whenever $(a_i : i < \omega)$ is a Morley sequence in a complete
  Lascar strong type $p$ then $\Cb(p) \in \bdd(a_i)$ for some (every)
  $i$.
\end{dfn}

\begin{lem}
  A cat $T$ is one-based if and only if, whenever $(a_i : i < \omega)$ is
  an indiscernible sequence, then $(a_i : 0 < i < \omega)$ is independent
  over $a_0$.
\end{lem}
\begin{proof}
  Remember that every indiscernible sequence is a Morley sequence over
  some set $A$: for example, a copy of the sequence.
  Setting $c = \Cb(a_i/A)$, $(a_i)$ is a Morley sequence over $c$.
  \\
  If $T$ is one based, then we have $c \in \bdd(a_0)$, so $a_i \ind_c
  a_{<i} \Longrightarrow a_i \ind_{a_0} a_{[1,i-1]}$ for every $i$.  Conversely, if
  $(a_i)$ is a Morley sequence in some Lascar strong type $p$ and $c =
  \Cb(p)$, then $c \in \dcl(a_{\geq2})$ so $a_1 \ind_c a_0 \Longrightarrow c =
  \Cb(a_1/ca_0)$ and $a_1 \ind_{a_0} a_{\geq2} \Longrightarrow a_1 \ind_{a_0} c \Longrightarrow c
  \in \bdd(a_0)$.
\end{proof}

\begin{prp}
  If $T$ is one-based then so is $T^\fP$.
\end{prp}
\begin{proof}
  Let $(a^j : j < \omega)$ be an indiscernible sequence in $U^\fP$.
  Extend $(a^j : j < \omega)$ to a very long $(a^j : j < \lambda)$, and take
  $(a_i^{<\lambda} : i < \omega) \models \mcl(a^{<\lambda})$.
  Considering it rather as a long sequence $(a_{<\omega}^j : j < \lambda)$, we
  may extract an indiscernible sequence, which shows that there are
  $(a_i^{<\omega} : i < \omega) \models \mcl(a^{<\omega})$ such that $(a_{<\omega}^j : j <
  \omega)$ is indiscernible.
  Since $T$ is one-based, the sequence $(a_{<\omega}^j : 0 < j < \omega)$ is a
  Morley sequence over $a_{<\omega}^0$, whereby $(a^j : 0 < j < \omega)$ is a
  Morley sequence over $a^0$.
\end{proof}

\section{The description of $T^\fP$ in $T$ and its functoriality}

In the first section we constructed the abstract elementary category
$\fP$, defined the language $\Delta^\fP$, and proved that $T^\fP =
\Th_{\Pi^\fP}(\fP)$ is a positive Robinson theory in $\Delta^\fP$, with
$\tS(T^\fP) = \tS(\fP)$.
In the topology on $\tS(T^\fP)$, closed sets are those defined by
partial $\Delta^\fP$-types, and equipped with this topology it is a
compact type-space functor.
However, this topology could have been obtained more directly, using
the categoric point of view described in \cite{pezz:fnctcat}.

Recall that we defined $\tS^{ind}_\alpha(T)$ as the subset of
$\tS_{\alpha\times\omega}(T)$ which consists of types of indiscernible sequence of
$\alpha$-tuples.
If $f : \alpha \to \beta$ is a map, and $f_{\times\omega} : \alpha \times \omega \to \beta \times \omega$ is its
natural extension to $\omega$-tuples, then $f_{\times\omega}^* : \tS_{\beta\times\omega}(T) \to
\tS_{\alpha\times\omega}(T)$ restricts to a map ${f^{ind}}^* : \tS^{ind}_\beta(T) \to
\tS^{ind}_\alpha(T)$, so $\tS^{ind}(T)$ is a sub-functor of
$\tS_{\times\omega}(T)$.

\begin{lem}
  \label{lem:descdef}
  \begin{enumerate}
  \item For every $\alpha$ there is a unique map $\Fd_\alpha : \tS_\alpha^{ind}(T) \to
    \tS_\alpha(T^\fP)$ satisfying $\Fd_\alpha(q) = p \Longleftrightarrow q \in
    \mcl(p)$.
  \item For closed sets $F \subseteq \tS^{ind}_\alpha(T)$ and $F' \subseteq
    \tS_\alpha(T^\fP)$, we have $\Fd_\alpha(F) = R_F$ and $\Fd_\alpha^{-1}(F') =
    \mcl(F')$, and these sets are closed.
    In particular, every $\Fd_\alpha$ is continuous and closed.
  \item \label{item:sq}
    Let $f : \alpha \to \beta$ be a map.
    Then the following diagram commutes, which makes $\Fd :
    \tS^{ind}(T) \to \tS(T^\fP)$ a morphism of functors:
    \begin{gather*}
      \begin{gathered}[c]
        \xymatrix{\tS^{ind}_\beta(T) \ar[r]^{\Fd_\beta} \ar[d]_{{f^{ind}}^*} &
          \tS_\beta(\fP) \ar[d]^{f^*}
          \\
          \tS^{ind}_\alpha(T) \ar[r]_{\Fd_\alpha} & \tS_\alpha(\fP)}
      \end{gathered}
    \end{gather*}
    Moreover, if $p \in \tS_\alpha(T^\fP)$, $q \in
    \Fd_\alpha^{-1}(p) = \mcl(p) \subseteq \tS^{ind}_\alpha(T)$ and $p' \in
    {f^*}^{-1}(p) = f_*(p) \subseteq \tS_\beta(T^\fP)$,
    then there is $q' \in \mcl(p') \cap f^{ind}_*(\mcl(p)) \subseteq
    \tS^{ind}_\beta(T)$.
  \item $\Fd : \tS^{ind}(T) \to \tS(T^\fP)$ is a quotient map, meaning
    that $\Fd$ is a surjective map, and the topology on $\tS(T^\fP)$ is
    maximal such that $\Fd$ is continuous.
  \end{enumerate}
\end{lem}
\begin{proof}
  \begin{enumerate}
  \item For every $q \in \tS_\alpha^{ind}(T)$ there is at most one value
    that $\Fd_\alpha$ can take, since $p \neq p' \Longrightarrow \mcl(p) \cap \mcl(p') =
    \emptyset$.
    Such a value always exists, as can be seen by applying
    \prettyref{lem:cncrind} with $\beta = 1$.
  \item This is just what \prettyref{lem:Rmcl} says.
  \item It is a fact that if $a$, $b$ and $P$ are given, then a
    sequence $(a_i : i < \omega)$ is a Morley sequence in $\tp(a/P)$ if
    and only if there are $b_{<\omega}$ such that $(a_ib_i : i < \omega)$
    is a Morley sequence in $\tp(ab/P)$.
    Then commutativity is one direction, and the moreover part is the
    other.
  \item Each $\Fd_\alpha$  is surjective since $\mcl(p) \neq \emptyset$ for every $p
    \in \tS_\alpha(\fP)$.
    A surjective, closed and continuous map is a quotient map.
  \end{enumerate}
\end{proof}

Thus, we could have defined the topology on $\tS(\fP)$ from the
beginning as the quotient topology, without ever bothering to define a
language explicitly.
Then, the commutativity statement in
\prettyref{lem:descdef}.\ref{item:sq} shows that
$f^* : \tS_\beta(\fP) \to \tS_\beta(\fP)$ is continuous, and the moreover
part shows that $f^*$ is closed.
As every set $\mcl(p)$ is closed, the topology on $\tS(\fP)$ is $T_1$,
and it is compact as the quotient of a compact topology.

In short, we could have skipped everything that comes after
\prettyref{cor:mcleq}, and still conclude that $\tS(\fP)$ is a compact
type-space functor, so there is a positive Robinson theory $T^\fP$ in
\emph{some} language such that $\tS(\fP) = \tS(T^\fP)$, but this time
also as topological functors.
In fact, we could have skipped the entire first section, constructing
$\tS(\fP)$ directly as the quotient of $\tS^{ind}(T)$ by the
appropriate equivalence relation (but then, of course, we wouldn't
know what it is that we are constructing).

This very abstract approach still seems (at least to the author) quite
convenient, and allows a few elegant observations.
Recall from \cite{pezz:fnctcat}:

\begin{dfn}
  \label{dfn:desc}
  \begin{enumerate}
  \item Let $\alpha_\Fd$ be an ordinal and $S$, $S'$ compact type-space
    functors.
    Let $\Fd : S_{\times\alpha_\Fd} \dashrightarrow S'$ be a continuous
    partial map, meaning that $(S_{\times\alpha_\Fd})_I = S_{I\times\alpha}$, $\dom(\Fd)
    \subseteq S_{\times\alpha_\Fd}$ is a closed sub-functor, and
    $\Fd : \dom(\Fd) \to S'$ is a continuous surjective morphism
    of functors.
    If $\varphi(x_{<n})$ is a formula in the language of $S'$ identify it with
    the closed set it defines $\varphi \subseteq S'_n$, and let $\tilde \varphi(\bar
    x_{<n})$ be the partial type in the language of $S$ defining
    $\Fd_n^{-1}(\varphi) \subseteq S_{n\times\alpha_\Fd}$.
    \\
    Let $x_{<n+m}$ be a tuple of variables, and let:
    \begin{align*}
      \psi(x_{<n}) & = \exists x_{\in[n,n+m)} \, \bigwedge_{j < l}
      \varphi_j(x_{i_{j,0}},\ldots,x_{i_{j,k_j-1}})
      \\
      \hat \psi(\bar x_{<n}) & = \exists \bar x_{\in[n,n+m)} \, \bigwedge_{j < l}
      \tilde \varphi_j(\bar x_{i_{j,0}},\ldots,\bar x_{i_{j,k_j-1}}),
    \end{align*}
    where each $\varphi_j$ is an $k_j$-ary formula, and $i_{j,s} < m+n$ for
    $j < l$ and $s < k_j$.
    \\
    Then $(\Fd,\alpha_\Fd)$ is a \emph{description}
    of $S'$ in $S$,  written $\Fd : S \dashrightarrow S'$,
    if whenever $\psi,\hat \psi$ are as above and $p \in \dom(\Fd)$ is in
    the right number of variables then:
    \begin{gather}
      \label{eq:desc}
      p \vdash \hat \psi \Longleftrightarrow \Fd(p) \vdash \psi
    \end{gather}
  \item A description is \emph{closed} if $\Fd$ is a closed map.
  \item If $T$ and $T'$ are simple cats, then a morphism $\Fr : \tS(T)
    \to \tS(T')$ \emph{preserves independence} if whenever $a,b,c$ and
    $a',b',c'$ are possibly infinite tuples in the universal domains
    of $T$ and $T'$, respectively, and $\tp^{T'}(a',b',c') =
    \Fr(\tp^T(a,b,c))$, then $b \ind_a^T c \Longleftrightarrow b' \ind_{a'}^{T'} c'$.
  \end{enumerate}
\end{dfn}

\begin{lem}
  \label{lem:inddesc}
  Let $q_j(x^{<\omega}_{<k_j})$ be partial types for $j < l$, each of
  which implying that
  $(x^s_{<k_j} : s < \omega)$ is an indiscernible sequence of $k_j$-tuples
  (in other words, $q_j \vdash \dom(\Fd_{k_j})$).
  \\
  Assume that $\pi(y^{<\omega}_{<n}) = \bigwedge_{j < l} q_j(y^{<\omega}_{i_{j,0}},
  \ldots, y^{<\omega}_{i_{j,k_j-1}})$ is consistent, where $i_{j,t}
  < n$ for every $j < l$ and $t < k_j$.
  Then it can be realised in $\dom(\Fd)$, that is to say that it has a
  realisation which is an indiscernible sequence of $n$-tuples.
\end{lem}
\begin{proof}
  Fix a very big $\lambda$, and let $\pi'(y^{<\lambda}_{<n})$ say that
  $\pi(y^{<\omega}_{<n})$, and in addition for every $j < l$ the sequence
  $(y^s_{i_{j,<k_j}} : s < \lambda)$ is indiscernible.
  Then $\pi'$ is consistent by compactness, and let $a^{<\lambda}_{<n}
  \models \pi'$.
  By indiscernibility, we have $a^{s_{<\omega}}_{<n} \models \pi$ for every
  increasing sequence $s_0 < s_1 < \cdots < \lambda$.
  As we took $\lambda$ sufficiently big, we can extract an indiscernible
  sequence $(b^s_{<n} : s < \omega)$ such that, for every $t < \omega$ there
  are $s^t_0 < \cdots < s^t_{t-1} < \lambda$ such that $b^{<t}_{<n} \equiv a^{s^t_{<t}}_{<n}$,
  whereby $b^{<\omega}_{<n} \models \pi$ as required.
\end{proof}

\begin{thm}
  \label{thm:desc}
  \begin{enumerate}
  \item The map $\Fd : \tS^{ind}(T) \to \tS(T^\fP)$, viewed as a
    partial map $\Fd : \tS_{\times\omega}(T) \dashrightarrow \tS(T^\fP)$, is a closed
    description also noted $\Fd : T \dashrightarrow T^\fP$, with a factor
    $\alpha_\Fd = \omega$, and domain $\dom(\Fd) = \tS^{ind}(T)$.
  \item This description is functorial: if $g : T \to T'$ is any
    morphism of type-space functors of thick simple cats, then there is
    a unique morphism $g^\fP : T^\fP \to {T'}^\fP$ that makes the
    following diagram commute:
    \begin{gather*}
      \xymatrix{T \ar[r]^g \ar@{-->}[d]_{\Fd_T} & T' \ar@{-->}[d]^{\Fd_{T'}}
        \\
        T^\fP \ar@{..>}[r]^{g^\fP} & {T'}^\fP}
    \end{gather*}
    \item If $g$ preserves independence, then so does $g^\fP$.
  \end{enumerate}
\end{thm}
\begin{proof}
  \begin{enumerate}
  \item Given \prettyref{lem:descdef}, all that is left to prove is
    \prettyref{eq:desc} of \prettyref{dfn:desc}.
    \\
    $\Longleftarrow$ follows from the moreover part of
    \prettyref{lem:descdef}.\ref{item:sq}.
    $\Longrightarrow$ follows from \prettyref{lem:inddesc}.
  \item Just verify that if $q,q' \in \tS^{ind}_\alpha(T)$ belong to the
    same Morley class, then so do $g(q),g(q')$.
  \item This is immediate from \prettyref{prp:indPeq}.
  \end{enumerate}
\end{proof}

We prove in \cite{pezz:fnctcat} that a theory describable in a
simple theory is simple.
Thus, had we taken the course proposed in the beginning of this
section, we could have concluded that $T^\fP$ is simple immediately,
even without giving an explicit characterisation of independence.

Recall also from \cite{pezz:fnctcat}:
\begin{dfn}
  Let $T$ be a simple cat and $T'$ a stable one.
  Then a \emph{stable representation} of $T$ in $T'$ is a morphism
  $\Fr : \tS(T) \to \tS(T')$ satisfying the following additional
  condition (called \emph{preservation of independence}):
  If $a,b,c$ are (possibly infinite) tuples in a model of $T$,
  $a',b',c'$ in a model of $T'$ and $\tp^{T'}(a',b',c') =
  \Fr(\tp^T(a,b,c))$ then $a \ind_b c \Longleftrightarrow a' \ind_{b'} c'$.
  \\
  With a minor abuse of notation we may also write it as $\Fr : T \to
  T'$.
\end{dfn}

\begin{cor}
  Assume that $T$ is simple and thick and has a thick stable
  representation (that is $\Fr : T \to T'$ where $T'$ is stable and
  thick).
  Then so does $T^\fP$.
\end{cor}
\begin{proof}
  Let $\Fr : T \to T'$ be a stable representation.  Then $\Fr^\fP : T^\fP
  \to {T'}^\fP$ preserves independence and ${T'}^\fP$ is stable, so it
  is a stable representation.
\end{proof}

Lastly, we would like to relate the lovely pairs construction with
another ``standard'' construction, namely the addition of a generic
automorphism.

\begin{dfn}
  For stable $T$, we let $\cC^A(T)$ denote the category of boundedly
  closed sets from $T$ equipped with an automorphism $\sigma$.
  \\
  If the category $\cC^A(T)$ forms a simple cat in a language
  extending that of $T$, and whose notion of
  independence is independence in $T$ of $\sigma$-closures, then we denote
  this cat by $T^A$ and say that \emph{$T^A$ exists}.
\end{dfn}

By \cite{pil:ec}, if $T$ is first order then $T^A$ exists, and it
can be further shown to be Robinson.
We do not wish to address here the issue of existence of $T^A$ in the
general case, so we will just assume that $T^A$ exists.
We do know however from \cite{pezz:fnctcat} that if $T^A$ exists
then we have a stable representation $\Fr_A : T^A \to T_{\times\omega}$, which
sends the type of an element in $T^A$ to the type in $T$ of its orbit
under the automorphism.

\begin{prp}
  Assume that $T^A$ does exist as and is thick.
  Then $(T^\fP)^A$ exists and is equal to
  $(T^A)^\fP$, and we have a commutative diagram, where $\Fr_A$ and
  $\Fr^\fP_A$ are the stable representations of $T^A$ and $T^\fP_A$,
  respectively:
  \begin{gather*}
    \xymatrix{T^A \ar[rr]^{\Fr_A} \ar@{-->}[d] & & T_{\times\omega} \ar@{-->}[d]
      \\
      (T^\fP)^A = (T^A)^\fP \ar[rr]^{\Fr^\fP_A} & & T^\fP_{\times\omega}}
  \end{gather*}
\end{prp}
\begin{proof}
  Let $\cC^\fP(T^A)$ be the category of pairs in $T^A$.
  Let $(A,\sigma) \in \cC^A(T^\fP)$.
  Then $A$ is boundedly closed in the sense of $T^\fP$, which means
  that $A^c \subseteq \dcl(P(A))$ (i.e., $A \ind_{P(A)} P$), and both $A$ and
  $P(A)$ are boundedly closed in the sense of $T$.
  Writing it as $(A,\sigma,P)$ it can also be viewed as a pair
  in the sense of $T^A$ and therefore an object of $\cC^\fP(T^A)$.
  \\
  This mapping from $\cC^A(T^\fP)$ into $\cC^\fP(T^A)$
  is a full and faithful functor:
  indeed, if $f : (A,P,\sigma) \to (B,P,\sigma)$ is a mapping, then it is a
  morphism in the sense of either category if and only if
  $f(P(A)) = P(f(A))$, $f\circ\sigma_A = \sigma_B\circ f$ and $f(A) \ind_{f(P(A))} P(B)$
  (since $\sigma$ is an automorphism of $A$, $P(A)$ and $P(B)$,
  independence in the sense of $T$ and of $T^A$ is the same).
  Moreover, this functor is co-final: every object of $\cC^\fP(T^A)$
  embeds into the image of an object of $\cC^A(T^\fP)$.
  \\
  This means that since $\cC^\fP(T^A)$ is an abstract elementary
  category so is $\cC^A(T^\fP)$, and they have the same type-spaces.
  Therefore these two categories are equivalent for our purposes:  we
  can use the language we chose for $\cC^\fP(T^A)$ also for
  $\cC^A(T^\fP)$, and both have the same positive Robinson theory in
  this language $(T^A)^\fP = (T^\fP)^A$.
  Finally it is an easy exercise to see that independence in the sense
  of $(T^A)^\fP$ coincides with independence in the sense of $T^\fP$
  of the $\sigma$-closures.
  \\
  The commutativity of the diagram is also easy.
\end{proof}

\section{Lowness and negation}

\begin{dfn}
  \label{dfn:low}
  \begin{enumerate}
  \item  We say that a formula $\varphi$ is \emph{clopen} if it defines a clopen
    set in the type-space.  Equivalently, if $\lnot\varphi$ is equivalent to a
    positive formula (and then we identify them).
  \item We recall that a \emph{$k$-inconsistency witness} for a
    formula $\varphi(x,y)$ is a formula $\psi(y_{<k})$ such that $\psi(y_{<k}) \land
    \bigwedge_{i<k} \varphi(x,y_i)$ is inconsistent.
    \\
    A formula $\varphi(x,y)$ is \emph{low} if it has a
    $k$-inconsistency witness $\psi$ such that, for every indiscernible
    sequence $(a_i)$, $\{\varphi(x,a_i)\}$ is inconsistent if and only if
    $\models \psi(a_0, \ldots, a_{k-1})$ (in other words, it has a universal
    inconsistency witness for indiscernible sequences).
  \item $T$ is \emph{low} if every formula is.
  \end{enumerate}
\end{dfn}

We recall that a cat $T$ is Robinson if and only if the
type-spaces are totally disconnected if and only if we can choose the
language such that all basic formulas are clopen.  It is first order
if and only if existential formulas are clopen as well.

\begin{rmk}
  If $T$ is first order then $\varphi(x,y)$ is low if and only if there is
  $k < \omega$ such that, if $(a_i : i < \omega)$ is indiscernible and
  $\{\varphi(x,a_i)\}$ is inconsistent then it is $k$-inconsistent.  The
  proofs of several of the results below can be simplified
  accordingly.  However, note that being Robinson does not
  suffice, since we need the negation of $\exists x \; \bigwedge_{i < k}
  \varphi(x,y_i)$, and this is an existential formula.
\end{rmk}

\begin{ntn}
  Let $\varphi(x,y)$ be a $T$-formula.  Then $\varphi(P,y) = \exists x \in P \;
  \varphi(x,y)$ is positive.
  \\
  If $\psi$ is a $k$-inconsistency witness for $\varphi$ then $\lnot_\psi\varphi(P,x) =
  R_\psi(y)$ is positive as well.
  \\
  If the existential formula $\psi(\bar y) = \exists x\; \bigwedge_{i < k}
  \varphi(x,y_i)$ defines a clopen set, then we write $\lnot_k\varphi(P,y) =
  R_{\lnot\psi}(y)$.
\end{ntn}

\begin{lem}
  If $A \subseteq B$ and $a \ind_A B$ then $\varphi(x,a)$ divides over $A$ if and
  only if it divides over $B$.
\end{lem}
\begin{proof}
  A Morley sequence for $a$ over $B$ is also a Morley sequence over
  $A$.
\end{proof}

\begin{lem}
  If $a$ and $b$ are tuples in $U^\fP$ satisfying precisely the same
  $\varphi(P,y)$ formulas, then $a \equiv^\fP b$.
\end{lem}
\begin{proof}
  Let $A \subseteq P$ be such that $a \ind_A P$.  Since $b$ satisfied all
  $\varphi(x,P)$ predicates that $a$ does, then by compactness, we can find
  $B \subseteq P$ such that $aA \equiv bB$.  Assume now that $\models \varphi(c,b)$ for
  some $c \in P$.  Then $\models \varphi(P,b) \Longrightarrow \models \varphi(P,a)$, whereby $\varphi(x,a)$
  does not divide over $P$ and therefore neither over $A$ (since $a
  \ind_A P$).  Then $\varphi(x,b)$ does not divide over $B$ either and $b
  \ind_B P$.  This suffices to see that $a \equiv^\fP b$.
\end{proof}

\begin{lem}
  Let $a$ be a tuple in $U^\fP$ and $\varphi(x,y)$ a $T$-formula.  Then
  $\varphi(x,a)$ does not divide over $P$ if and only if it is satisfied in
  $P$.
\end{lem}
\begin{proof}
  If $\varphi(x,a)$ is realised in $P$, clearly it cannot divide over $P$.
  Conversely, assume that it does not divide over $P$.  Then there is
  a complete type $p \in \tS(aP)$ such that $p(x) \vdash \varphi(x,a)$ and $p$
  does not divide over $P$.  By loveliness of the universal domain, we
  can realise $p$ in $P$.
\end{proof}

\begin{cor}
  $\fP \models \lnot\varphi(P,y) \leftrightarrow \bigvee_\psi \lnot_\psi\varphi(P,y)$, where $\psi$ varies over all
  inconsistency witnesses for $\varphi$.
\end{cor}
\begin{proof}
  $\varphi(x,a)$ is not satisfied in $P$ if and only if it divides over $P$
  if and only if there is a Morley sequence $(a_i : i < \omega)$ for $a$
  over $P$ such that $\{\varphi(x,a_i)\}$ is inconsistent if and only if
  there is $(a_i : i < \omega) \models \mcl(a)$ satisfying an inconsistency
  witness for $\varphi$.
\end{proof}

\begin{cor}
  If $a$ and $b$ are tuples in $\fP$, and $b$ satisfies every formula
  of the form $\varphi(P,y)$ or $\lnot_\psi\varphi(P,y)$ that $a$ does, then $a \equiv^\fP
  b$.
\end{cor}
\begin{proof}
  In this case $a$ and $b$ satisfy precisely the same $\varphi(P,y)$
  formulas.
\end{proof}

\begin{lem}
  A formula $\varphi(x,y)$ is low if and only if for every $\lambda$ there is a
  partial type $\Phi^{\mathrm{div}}_\varphi(y,Z)$, $|Z| = \lambda$, such that
  $\varphi(x,a)$ divides over a set $B$ of cardinality $\lambda$ if and only if $\models
  \Phi^{\mathrm{div}}_\varphi(a,B)$.
\end{lem}
\begin{proof}
  Assume that $\varphi$ is low, and let $\psi$ be the universal inconsistency
  witness.  Then the partial type saying that there is a
  $Z$-indiscernible sequence $(y_i : i < \omega)$ satisfying $y_0 = y \land
  \psi(y_0, \ldots, y_{k-1})$ will do.
  \\
  For the converse, write the partial type saying that $(y_i : i \leq
  \omega)$ is indiscernible, $\exists x \; \bigwedge_{i < k} \varphi(x,y_i)$ for every $k <
  \omega$, and $\varphi(x,y_\omega)$ divides over $y_{<\omega}$.  If this could be
  realised, we could continue the sequence to length $2\omega$, in which
  case $(y_i : \omega \leq i < 2\omega)$ would be a Morley sequence over
  $y_{<\omega}$, whereby $\varphi(x,y_\omega)$ cannot divide over $y_{<\omega}$.  Then
  this is inconsistent, so there are $k_0 < \omega$, $\psi_0$ implied by the
  statement that the sequence is indiscernible, and $\psi_1(\bar y) \in
  \Phi^{\mathrm{div}}_\varphi(y_\omega,y_{<\omega})$, such that $\psi_0(\bar y) \land
  \psi_1(\bar y) \land \exists x \; \bigwedge_{i < k_0} \varphi(x,y_i)$ is contradictory.
  Let $k$ be the total number of $y_i$ appearing there, and write
  $\psi_0 \land \psi_1 = \psi(y_0, \ldots, y_{k-1})$.  Then $\psi$ is a
  $k$-inconsistency witness for $\varphi$, and we claim that it is
  universal.  Indeed, assume that $(a_i : i < \omega)$ are indiscernible
  and $\{\varphi(x,a_i)\}$ is inconsistent.  Let $a_\omega$ continue this
  sequence, so $\varphi(x,a_\omega)$ divides over $a_{<\omega}$: then $\psi_0$ holds
  due to the indiscernibility, and $\psi_1$ since $\models
  \Phi^{\mathrm{div}}_\varphi(a_\omega,a_{<\omega})$.  This shows that $\psi$ witnesses
  that $\varphi$ is low.
\end{proof}

\begin{rmk}
  The converse part was first proved in a special case by Vassiliev.
\end{rmk}

\begin{cor}
  If $\varphi$ is low, then $\varphi(P,y)$ is clopen in $T^\fP$.  The converse
  holds if $T$ is Robinson.
\end{cor}
\begin{proof}
  For left to right, if $\psi(\bar y)$ witnesses that $\varphi(x,y)$ is low,
  then $\fP \models \lnot_\psi\varphi(P,y) \leftrightarrow \lnot\varphi(P,y)$.
  \\
  For the converse, assume that $T$ is Robinson, and that all the
  formulas are clopen.  Set $\Phi^{\mathrm{div}}_\varphi(y,Z) = \bigcap \{\tp(a,B)
  : \varphi(x,a) \text{ divides over } B\}$, and it will be enough to show
  that $\models \Phi^{\mathrm{div}}_\varphi(a,B) \Longrightarrow \varphi(x,a)$ divides over $B$.
  \\
  Assume then that $\models \Phi^{\mathrm{div}}_\varphi(a,B)$, and set $q(y,Z) =
  \tp(a,B)$.  For every formula $\chi$ we have $\chi \in q$ if and only if
  $\lnot\chi \notin q$ if and only if there are $a',B'$ such that $\models \chi(a',B')$
  and $\varphi(x,a')$ divides over $B'$.  We can realise $B'$ in $P$ and
  then realise $a'$ such that $a' \ind_{B'} P$.  Then $\varphi(x,a')$
  divides over $P$, and $\models \lnot\varphi(P,a')$.
  \\
  This shows that $\lnot\varphi(P,y) \land Z \subseteq P \land q(y,Z)$ is finitely
  consistent.  As $\lnot\varphi(P,y)$ is positive, this is consistent, and we
  might just as well assume that it is realised by $a,B$.  But then
  $\lnot\varphi(P,a) \Longrightarrow \varphi(x,a)$ divides over $B$.
\end{proof}

\begin{cor}
  $T$ is low if and only if every formula $\varphi(P,y)$ is clopen.  In
  this case $T$ is first-order, and the formulas $\varphi(P,y)$, $\lnot\varphi(P,y)$
  form a basis for the $\cL^\fP$, so taking them as basic formulas
  $T^\fP$ is Robinson.
\end{cor}
\begin{proof}
  If $T$ is low then we know that every formula $\varphi(P,y)$ is clopen.
  Conversely, we know that $P$ is a model of $T$, so if every formula
  $\varphi(P,y)$ is clopen then $T$ is first order (existential formulas
  are clopen), and then we know that $T$ is low.
\end{proof}

\begin{rmk}
  If fact, when $T$ is low with quantifier elimination, we can
  axiomatise $T^\fP$ directly as a universal Robinson theory
  in the language $\cL'$ consisting of predicates $\varphi(P,y)$ for every
  formula $\varphi(x,y)$ in the language of $T$:
  \\
  For every $n,m < \omega$ and formulas $\varphi_i(x_i,y_i)$ for $i < n$ and
  $\psi_j(t_j,z_j)$ for $j < m$, consider the statement:
  \begin{gather*}
    \forall y_{<n} z_{<m} \, [\bigwedge_{i < n} \varphi_i(P,y_i) \land \bigwedge_{j < m}
    \lnot\psi_j(P,z_j)] \to \exists x_{<n} \; \bigwedge_{i < n} \varphi_i(x_i,y_i) \land \bigwedge_{j <
      m} \Phi^{\mathrm{div}}_{\psi_j}(z_i,x_{<n})
  \end{gather*}
  Since $T$ is assumed to have quantifier elimination, the statement
  $\exists x_{<n} \; \bigwedge_{i < n} \varphi_i(x_i,y_i) \land \bigwedge_{j < m}
  \Phi^{\mathrm{div}}_{\psi_j}(z_i,x_{<n})$ is equivalent modulo $T$ to a
  quantifier-free partial type.  Therefore, the statement above can be
  viewed as a universal theory in $\cL'$.  Take $T'$ to be the
  universal theory consisting of all universal $\cL'$-sentences thus
  obtained.  Then $T'$ is a Robinson theory, equivalent as a
  cat to $T^\fP$ (that is, has the same type-space).
  \\
  This is proved in \cite{ppv:pairs}.
\end{rmk}

We know that a stable theory is low if and only if it is first-order:
one direction is classical, the other was proved above.  We can also
prove:

\begin{prp}
  If $T$ is stable and Robinson then so is $T^\fP$.
\end{prp}
\begin{proof}
  Since $T$ is stable, $\mcl(p)$ is a complete type for every $p$,
  whereby $\lnot R_\varphi = R_{\lnot\varphi}$ for every formula $\varphi$, and $T^\fP$ is
  Robinson.
\end{proof}

\begin{qst}
  Find a necessary and sufficient condition for $T^\fP$ to be
  Robinson.
\end{qst}

\providecommand{\bysame}{\leavevmode\hbox to3em{\hrulefill}\thinspace}
\providecommand{\MR}{\relax\ifhmode\unskip\space\fi MR }
\providecommand{\MRhref}[2]{%
  \href{http://www.ams.org/mathscinet-getitem?mr=#1}{#2}
}
\providecommand{\href}[2]{#2}

\end{document}